\documentclass{article}

\usepackage{verbatim}
\usepackage{amsmath, amsfonts, amsthm,amscd,amssymb}
\usepackage{hyperref}
\usepackage{graphicx}
\usepackage{tikz,tikz-cd}
\usepackage{todonotes}

\theoremstyle{plain}
\newtheorem{thm}{Theorem}
\newtheorem{lem}[thm]{Lemma}

\newtheorem{cor}[thm]{Corollary}

\theoremstyle{definition}

\newtheorem{axiom}[thm]{Axiom}

\numberwithin{thm}{section}

\newcommand{\pd}{\partial}

\DeclareMathOperator{\Fix}{Fix}

\DeclareMathOperator{\ind}{ind}

\def\R{{\mathbb R}}

\title{Axioms for the fixed point index of $n$-valued maps, and some applications}
\author{P. Christopher Staecker}

\newcommand{\adm}{\mathcal{C}}

\begin{document}
\bibliographystyle{hplain}

\maketitle

\begin{abstract}
We give an axiomatic characterization of the fixed point index of an $n$-valued map. For $n$-valued maps on a polyhedron, the fixed point index is shown to be unique with respect to axioms of homotopy invariance, additivity, and a splitting property. This uniqueness is used to obtain easy proofs of an averaging formula and product formula for the index. In the setting of $n$-valued maps on a manifold, we show that the axioms can be weakened.
\end{abstract}

\section{Introduction}
Given a set $X$ and a positive integer $n$, an $n$-valued function is a set-valued function $f$ on $X$ such that $f(x)$ has cardinality exactly $n$ for every $x\in X$. Equivalently, an $n$-valued function on $X$ is a single-valued function $f:X \to D_n(Y)$, where $D_n(Y)$ is the \emph{unordered configuration space of $n$ points in $Y$}, defined as:
\[ D_n(Y) = \{ \{y_1,\dots,y_n\} \mid y_i \in Y, y_i \neq y_j \text{ for } i\neq j \}. \]

When $Y$ is a topological space, we give $D_n(Y)$ a topology as follows: begin with the product topology on $Y^n$, then consider the subspace of tuples $(y_1,\dots,y_n)\in Y^n$ with $y_i \neq y_j$ for $i\neq j$. Then $D_n(Y)$ is the quotient of these tuples up to ordering, and so its topology is given by the quotient topology. When $f:X\to D_n(Y)$ is continuous, we call it an \emph{$n$-valued map from $X$ to $Y$}.

Continuity of $n$-valued maps can also be defined in terms of lower- and upper-semicontinuity. These approaches are equivalent: see \cite{bg17}.

We say that an $n$-valued map $f:X\to D_n(Y)$ is \emph{split} if there are $n$ single valued continuous functions $f_1,\dots,f_n:X\to Y$ with $f(x) = \{ f_1(x),\dots,f_n(x) \}$ for every $x\in X$. In this case we write $f= \{f_1,\dots, f_n\}$. Not all $n$-valued maps are split. 

In this paper all spaces will be assumed to be finite polyhedra, and we will generally focus on $n$-valued selfmaps of the form $f:X \to D_n(X)$. A \emph{fixed point} of such a map is a point $x\in X$ with $x \in f(x)$.

Topological fixed point theory of $n$-valued maps was first studied by Schirmer in \cite{schi84a, schi84b}. Schirmer's approach is fundamentally based on the following lemma of Banach and Mazur \cite{bm34}:
\begin{lem}[Splitting Lemma]
If $X$ is simply connected, then any $n$-valued map $f:X\to D_n(Y)$ is split.
\end{lem}

The local fixed point index is a classical invariant of a self-map which gives an algebraic count of the ``multiplicity'' of a fixed point. We will write the classical local fixed point index of a single-valued map $f$ on the set $U$ as $\ind(f,U)$.

The index is typically defined geometrically or homologically, see \cite{brow71}. There is also a substantial literature taking an axiomatic approach. The fixed point index is the unique numerical homotopy invariant which satisfies certain properties. In \cite{fps04}, in the context of manifolds, it is shown that the index is unique with respect to axioms of homotopy, additivity, and a normalization for constant maps. This approach has been adapted to similar results for the coincidence index \cite{stae07a,gs12} and the Reidemeister trace \cite{stae09b}. Other axiom schemes have been used for the Lefshetz number, see \cite{ab04}. 

Schirmer in \cite{schi84a} defines the local fixed point index of an $n$-valued map. For clarity we write this index, when $f$ is an $n$-valued map, as $\ind_n(f,U)$. Schirmer's definition is as follows: First change $f$ by homotopy to some map $f'$ which has finitely many isolated fixed points $x_j$ inside $U$. Then choose a simply connected neighborhood $V_j\subset U$ around each fixed point. Then $f$ splits on $V_j$ as $f = \{f_1,\dots,f_n\}$, and $x_j$ will be a fixed point of exactly one of the $f_i$. Shirmer then defines $\ind_n(f,U) = \sum_{j}\ind(f_i,V_j)$, where $\ind$ denotes the classical fixed point index of a single-valued map.
Schirmer proves that the index is well defined and is homotopy invariant and additive on disjoint subsets.

The fixed point index of an $n$-valued map is also discussed by Crabb \cite{crab15}. Crabb in particular demonstrates an $n$-valued form of the Lefschetz-Hopf Theorem, that the fixed point index of $(f,X)$ equals the Lefschetz number, the alternating sum of the traces of some induced map in homology. The Lefschetz number of an $n$-valued map was defined by Brown in \cite{brow07}. If we denote the Lefschetz number of some $n$-valued map as $L(f)$, Proposition 3.5 and Corollary 5.9 of \cite{crab15} give:
\begin{thm}(Crabb, 2015) Let $f:X\to D_n(X)$ be an $n$-valued map. Then
\[ L(f) = \ind(f,X). \]
\end{thm}

In Section \ref{uniquenesssection} we prove a uniqueness theorem for the fixed point index of an $n$-valued map on a polyhedron. We show that $\ind_n$ is the unique function satisfying homotopy invariance, additivity on disjoint subsets, and a splitting property. As applications of our uniqueness theorem we prove in Sections \ref{averagingsection} and \ref{productsection} an averaging and product formula for the index. (The product formula also appears in Proposition 3.4 of \cite{crab15}.) In Section \ref{manifoldsection} we prove a stronger uniqueness theorem (using weaker axioms) in the setting of $n$-valued maps on manifolds. 

We would like to thank Michael Crabb for bringing the work in \cite{crab15} to our attention. 

\section{The axioms}\label{uniquenesssection}

In this section we prove a uniqueness theorem for the local fixed point index of an $n$-valued map using axioms inspired by \cite{fps04}. That paper uses the setting of maps on differentiable manifolds, but our strong third axiom will allow the more general setting of maps on any complex.

Given a complex $X$, we say a pair $(f,U)$ is \emph{admissible} when $f$ is an $n$-valued selfmap on $X$ with $U\subset X$ and $\Fix(f) \cap U$ is compact. (Schirmer in \cite{schi84b} defines $(f,U)$ as admissible when $\Fix(f) \cap \pd U = \emptyset$. The present definition is equivalent since $\Fix(f)$ is a closed set and the complex $X$ is compact.) A homotopy $H:X\times [0,1] \to X$ of some admissible pair $(f,U)$ to $(g,U)$ is admissible when $\{(x,t) \mid x \in H(x,t) \cap U \}$ is compact.

For some complex $X$, let $\adm_n(X)$ be the set of admissible pairs $(f,U)$ where $f$ is an $n$-valued selfmap on $X$. We will write $\adm(X) = \adm_1(X)$ for the single-valued case. In the following, let $\iota$ be a real-valued function $\iota:\adm_n(X) \to \R$. 
\begin{axiom}(Homotopy) \label{htpax}
Let $(f,U), (g,U)\in \adm_n(X)$ be admissibly homotopic. Then $\iota(f,U) = \iota(g,U)$. 
\end{axiom}

\begin{axiom}(Additivity) \label{addax}
Let $(f,U)\in \adm_n(X)$ with $\Fix(f) \cap U \subseteq U_1 \sqcup U_2$, where $\sqcup$ denotes disjoint union. Then $\iota(f,U) = \iota(f,U_1) + \iota(f,U_2)$.
\end{axiom}

\begin{axiom}(Splitting) \label{splittingax}
Let $(f,U)\in \adm_n(X)$ such that on the set $U$, the map $f$ splits as $\{f_1,\dots,f_n\}$, and $\Fix(f) \cap U$ is a single point. Then
\[ \iota(f,U) = \ind(f_i,U), \]
where $f_i$ is the unique map with $\Fix(f_i) \cap U \neq \emptyset$.
\end{axiom}

Schirmer proves each of these axioms are satisfied when $\iota$ is the local fixed point index of an $n$-valued map. The homotopy axiom is Theorem 4.5 of \cite{schi84b}, the additivity axiom is Theorem 4.6 of \cite{schi84b}, and the splitting axiom is an immediate consequence of the definition of the index, which is defined in terms of splittings.

An immediate consequence of the additivity axiom is the following property:
\begin{thm}\label{excision}[Excision property]
Let $\iota$ satisfy the additivity axiom. If $(f,U)\in \adm_n(X)$ and $V\subset U$ with $(f,V) \in \adm_n(X)$ and $\Fix(f)\cap U \subseteq V$, then $\iota(f,U) = \iota(f,V)$.
\end{thm}
\begin{proof}
First we prove an ``empty set property'', that $\iota(f,\emptyset) =0$. This is proved by writing the disjoint union $\Fix(f) \cap \emptyset \subseteq \emptyset \sqcup \emptyset$, and so
\[ \iota(f,\emptyset) = \iota(f,\emptyset) + \iota(f,\emptyset) \]
and thus $\iota(f,\emptyset) = 0$.

No we prove the excision property. Writing $\Fix(f) \cap U \subseteq \emptyset \sqcup V $ and applying the additivity axiom gives:
\[
\iota(f,U) = \iota(f,\emptyset) + \iota(f,V) = \iota(f,V). \qedhere
\]
\end{proof}

In the case when $\Fix(f)\cap U$ is empty, the excision property gives $\iota(f,U) = \iota(f,\emptyset) = 0$. Thus we obtain:

\begin{thm}\label{fpp}[Solution property]
Let $\iota$ satisfy the additivity axiom. If $(f,U)\in \adm_n(X)$ and $\iota(f,U) \neq 0, then $ $\Fix(f) \cap U$ is nonempty. 
\end{thm}

Recall Schirmer's definition of the index: starting with an admissible pair $(f,U)$, first we replace $f$ by an admissibly homotopic map $f'$ so that $\Fix(f')\cap U$ is finite. The existence of such a ``fix-finite'' homotopy is the main result of \cite{schi84a}. Then we choose simply connected isolating neighborhoods $V_j$ of the fixed points of $f'$ in $U$. Since $V_j$ are simply connected we have a splitting $f = \{f_1,\dots,f_n\}$ on $V_j$, and there is some $i$ such that $\Fix(f)\cap V_j = \Fix(f_i)\cap V_j$ is a single point. Then Schirmer defines:
\[ \ind_n(f,U) = \sum_{j} \ind(f_i,V_j). \]

Schirmer proves that this definition is independent of the various choices of $f'$ and the sets $V_j$. But we remark that Schirmer's definition of $\ind_n$ is constructed entirely using the properties of the Homotopy, Additivity, and Splitting axioms. That is, if $\iota$ is any function satisfying these three axioms, then following the same construction above we must also have
\[ \iota(f,U) = \sum_j \ind(f_i,V_j). \]

Thus we immediately obtain:
\begin{thm}\label{uniqueness}
Given any complex $X$, the local fixed point index is the unique function $\iota:\adm_n(X) \to \R$ obeying the Homotopy, Additivity, and Splitting axioms.
\end{thm}

Note that $\ind_n$ is proved to be the unique real-valued function satisfying the axioms. In fact the values of the fixed point index are always integers. 

\section{An averaging formula}\label{averagingsection}
In this section we obtain an $n$-valued analogue of the averaging formula for the fixed point index. The first averaging formula of this kind is apparently Jiang's averaging formula for the Lefschetz number which appears as Theorem III.2.12 of \cite{jian83}. When $f:X\to X$ is a single-valued selfmap and $p:\bar X \to X$ is a finite covering with covering group $G$, let $\bar f:\bar X \to \bar X$ be a lift of $f$. Then the averaging formula states:
\[ L(f) = \frac{1}{|G|}\sum_{\alpha\in G} L(\alpha \bar f). \]

This is a special case of an averaging formula for the local fixed point index, which is straightforward but does not seem to have appeared in the literature until recently. An axiomatic proof is given in \cite{ls17}, as well as a generalization to the coincidence Reidemeister trace.
\begin{thm}
Let $f:X\to X$ be a selfmap and $p:\bar X \to X$ is a finite covering with covering group $G$, let $\bar f:\bar X \to \bar X$ be a lift of $f$. If $(f,U) \in \adm(X)$, then $(\bar f,p^{-1}(U))\in \adm(\bar X)$ and 
\[ \ind(f,U) = \frac{1}{|G|}\sum_{\alpha\in G} \ind(\alpha \bar f, p^{-1}(U)). \]
\end{thm}

Our goal for this section is to prove a similar formula for $\ind_n$. 

Given a single valued map $g: X \to Y$, let $g^n$ be the function which applies $g$ to each element of an $n$-element set. (This function $g^n$ may not always be $n$-valued.)

When $\bar X \to X$ is a covering map, and $f:X\to D_n(X)$ is an $n$-valued map, say that an $n$-valued map $\bar f:\bar X \to D_n( \bar X)$ is a \emph{lifting} of $f$ when the following diagram commutes:
\[ \begin{tikzcd}
\bar X \arrow[r,"\bar f"] \arrow[d,"p"]&D_n(\bar X) \arrow[d,"p^n"]\\
X \arrow[r,"f"] & D_n(X)
\end{tikzcd}
\]

Let $G$ be the group of covering transformations for the covering $p:\bar X \to X$, and $\bar f:\bar X \to D_n(\bar X)$ be a lift of some $n$-valued map $f:X\to D_n(X)$. Viewing any $\alpha \in G$ as an automorphism of $\bar X$, the composition $\alpha^n \bar f:\bar X \to D_n(\bar X)$ is always an $n$-valued map, which is another lift of $f$. 

Our averaging formula takes the following form:
\begin{thm}
Let $p:\bar X \to X$ be a finite covering map, and let $G$ be the group of covering transformations. Let $f:X\to D_n(X)$ be an $n$-valued map, and let $\bar f: \bar X \to D_n(\bar X)$ be a lifting. Let $(f,U)\in \adm_n(X)$, and let $\bar U = p^{-1}(U)$. Then $(\alpha^n\bar f, \bar U) \in \adm_n(\bar X)$ for each $\alpha \in G$, and
\[ \ind_n(f,U) = \frac{1}{|G|} \sum_{\alpha \in G} \ind_n(\alpha^n \bar f,\bar U). \]
\end{thm}
\begin{proof}
Since $(f,U) \in \adm_n(X)$ and $\bar X$ is a finite cover of $X$, it is clear that $(\bar f,\bar U) \in \adm_n(\bar X)$, and thus that $(\alpha^n \bar f, \bar U) \in \adm_n(\bar X)$, since $\alpha^n:D_n(\bar X) \to D_n(\bar X)$ is a homeomorphism.

Given an admissible pair $(f,U)$, let $\iota$ be defined by:
\[ \iota(f,U) = \frac{1}{|G|} \sum_{\alpha \in G} \ind_n(\alpha^n \bar f,\bar U). \]
We must show that $\iota(f,U) = \ind_n(f,U)$, and we do this by showing that $\iota$ satisfies the Homotopy, Additivity, and Splitting axioms. 

For the Homotopy axiom, let $(f',U)$ be admissibly homotopic to $(f,U)$. Then the admissible homotopy from $(f,U)$ to $(f',U)$ naturally lifts to an admissible homotopy starting at $(\alpha^n\bar f, \bar U)$ and ending with $(\alpha^n\bar f', \bar U)$, where $\bar f'$ is some lifting of $f'$. Then by the homotopy property of $\ind_n$ we have:
\[
\iota(f,U) = \frac{1}{|G|} \sum_{\alpha \in G} \ind_n(\alpha^n \bar f,\bar U) 
=\frac{1}{|G|} \sum_{\alpha \in G} \ind_n(\alpha^n \bar f',\bar U) = \iota(f',U),
\]
and so $\iota$ satisfies the homotopy axiom.

For the Additivity axiom, let $\Fix(f)\cap U \subseteq U_1 \sqcup U_2$. Then let $\bar U_i = p^{-1}(U_i)$ for $i=1,2$, and we have $\Fix(\alpha^n \bar f)\cap \bar U \subseteq \bar U_1 \sqcup \bar U_2$. Thus the additivity property of $\ind_n$ gives:
\begin{align*} 
\iota(f,U) &= \frac{1}{|G|} \sum_{\alpha \in G} \ind_n(\alpha^n \bar f,\bar U) \\
&= \frac{1}{|G|} \sum_{\alpha \in G} (\ind_n(\alpha^n \bar f,\bar U_1) + \ind_n(\alpha^n \bar f,\bar U_2)) \\
&= \iota(f,U_1) + \iota(f,U_2). 
\end{align*}

For the Splitting axiom, say that $f$ splits on $U$ as $\{f_1,\dots,f_n\}$, and $\Fix(f) \cap U = \Fix(f_i) \cap U$ is a single point. Lifting the maps $f_i$, we obtain a splitting of $\bar f$ as $\bar f= \{\bar f_1,\dots, \bar f_n\}$.
First we claim that, for each $\alpha \in G$, we have $\ind_n(\alpha^n \bar f, \bar U) = \ind_n(\alpha \bar f_i, \bar U)$. In the case where $\Fix(\alpha \bar f) \cap \bar U = \emptyset$, then $\ind_n(\alpha^n \bar f, \bar U) = 0 = \ind_n(\alpha \bar f_i, \bar U)$ as desired.
In the case where $\Fix(\alpha^n \bar f) \cap \bar U$ is not empty, note that when $j\neq i$ then $\Fix(\alpha^n\bar f_j) \cap \bar U = \emptyset$, since otherwise $f_j$ would have a fixed point in $U$. Thus $\Fix(\alpha\bar f_i)\cap \bar U$ must be nonempty, and the splitting property of $\ind_n$ gives $\ind_n(\alpha^n \bar f, \bar U) = \ind(\alpha \bar f_i, \bar U)$.

Thus for each $\alpha$ we have $\ind_n(\alpha^n \bar f, \bar U) = \ind(\alpha \bar f_i, \bar U)$, and so:
\[
\iota(f,U) = \frac{1}{|G|} \sum_{\alpha \in G} \ind_n(\alpha^n \bar f,\bar U)  
=\frac{1}{|G|} \sum_{\alpha \in G} \ind(\alpha \bar f_i,\bar U) = \ind(f_i,U)
\]
where the last equality is the averaging formula for the classical fixed point index.
\end{proof}

Taking $U=X$ in the theorem above, we obtain an averaging formula for the Lefschetz number of an $n$-valued map:
\begin{cor}
Let $p:\bar X \to X$ be a finite covering map, and let $G$ be the group of covering transformations. Let $f:X\to D_n(X)$ be an $n$-valued map, and let $\bar f: \bar X \to D_n(\bar X)$ be a lifting. Then:
\[ L(f) = \frac{1}{|G|} \sum_{\alpha \in G} L(\alpha^n \bar f). \]
\end{cor}

\section{A product formula}\label{productsection}
In this section we prove a product formula for the index of an $n$-valued map, again using the axioms of Theorem \ref{uniqueness}. 

Our aim is to generalize the well known product formula for the index of a single valued map (see \cite{jian83}, I.3.5.iv):
\[ \ind(f\times g, U\times V) = \ind(f,U) \ind(g,V). \]

Given an $n$-valued map $f:X\to D_n(X)$ and an $m$-valued map $g:Y \to D_m(Y)$, the product $f\times g:X\times Y \to D_{mn}(X \times Y)$ is the $mn$-valued map given by:
\[ (f\times g)(x,y) = f(x) \times g(y). \]
The goal for this section is to prove the following product formula:
\[ \ind_{mn}(f\times g, U\times V) = \ind_n(f,U) \ind_m(g,V). \]
This formula appears as part of Proposition 3.4 of \cite{crab15}, where it is proved using entirely different methods. 

It is not clear exactly how to approach the product formula by directly using Theorem \ref{uniqueness}. Instead we prove a simple variation of Theorem \ref{uniqueness} specialized to the category of product mappings.

Given complexes $X$ and $Y$ and positive integers $n$ and $m$, let $\mathcal P_{n,m}(X,Y)$ be the set of all pairs $(f\times g,U\times V)$ where $f:X \to D_n(X)$ is an $n$-valued map and $g:Y \to D_n( Y)$ is an $m$-valued map, and $(f,U)\in \adm_n(X)$ and $(g,V)\in \adm_m(Y)$. 
\begin{thm}\label{productuniqueness}
The local fixed point index  $\ind_{mn}(f,U)$ for $mn$-valued maps is the unique real-valued function $\iota:\mathcal P_{n,m}(X,Y) \to \R$ satisfying the following axioms:
\begin{itemize}
\item (Product Homotopy) Let $(f,U)$ be admissibly homotopic to $(f',U)$ and $(g,V)$ be admissibly homotopic to $(g',V)$. Then $\iota(f\times g,U\times V) = \iota(f'\times g', U\times V)$. 
\item (Additivity) Let $(f\times g,U\times V) \in \mathcal P_{n,m}(X,Y)$ with $\Fix(f\times g) \cap (U\times V) \subset (U_1\times V_1) \sqcup (U_2\times V_2)$. Then $\iota(f\times g,U\times V) = \iota(f\times g,U_1\times V_1) + \iota(f\times g, U_2\times V_2)$. 
\item (Splitting) Let $(f\times g, U\times V) \in \mathcal P_{n,m}(X,Y)$ such that $f\times g$ splits on $U\times V$ as $(h_1,\dots, h_{mn})$, and $\Fix(f\times g) \cap U\times V = \Fix(h_i, U\times V)$ is a single point. Then
\[ \iota(f\times g,U\times V) = \ind(h_i,U\times V). \]
\end{itemize}
\end{thm}

\begin{proof}
Let $\iota:\mathcal P_{n,m}(X,Y)\to \R$ satisfy the three axioms, and let $(f\times g,U\times V) \in \mathcal P_{n,m}(X,Y)$. We will show that $\iota(f\times g,U\times V) = \ind_{mn}(f\times g, U\times V)$. 

Schirmer's fix-finite theorem of \cite{schi84b} applied to $f$ and $g$ provides admissible pairs $(f',U)$ and $(g',V)$ such that $(f,U)$ is admissibly homotopic to $(f',U)$ and $(g,V)$ is admissibly homotopic to $(g',V)$ and $\Fix(f')\cap U$ and $\Fix(g')\cap V$ are finite sets. By the product homotopy axiom, we have 
\[ \iota(f\times g,U\times V) = \iota(f'\times g', U\times V). \]

It is clear to see that $\Fix(f'\times g')\cap (U\times V) = (\Fix(f')\cap U) \times (\Fix(g')\cap V)$, and in fact if $U_1, \dots, U_k \subset U$ are isolating neighborhoods of the fixed points of $f'$, and $V_1,\dots, V_l\subset V$ are isolating neighborhoods of the fixed points of $g'$, then the various $U_i\times V_j$ are isolating neighborhoods of the fixed points of $f'\times g'$, and thus the additivity axiom gives:
\[ 
\iota(f\times g, U\times V) = \sum_{s=1}^k \sum_{t=1}^l \iota(f'\times g', U_s\times V_t).
\]

We may assume that the $U_s$ and $V_t$ are all simply connected, and so $U_s\times V_t$ is simply connected, and so $f'\times g'$ splits on each $U_s\times V_t$ as $(h_{1,s,t},\dots,h_{mn,s,t})$, and there is exactly one value of $i$ such that $\Fix(h_{i,s,t}) \cap U_s\times V_t$ is nonempty, and this fixed point set is a single point. By the splitting property we have $\iota(f'\times g', U_s\times V_t) = \ind(h_{i,s,t},U_s\times V_t)$. But for the same reasons, the splitting property of $\ind_{mn}$ gives $\ind_{mn}(f'\times g', U_s\times V_t) = \ind(h_{i,s,t},U_s\times V_t)$. Thus $\iota(f'\times g', U_s\times V_t) = \ind_{mn}(f'\times g', U_s\times V_t)$, and the above becomes:
\[
\iota(f\times g, U\times V) = \sum_{s=1}^k \sum_{t=1}^l \ind_{mn}(f'\times g',U_s\times V_t).
\]
Then the additivity and homotopy properties of $\ind_{mn}$ give:
\[ \iota(f\times g, U\times V) = \ind_{mn}(f'\times g', U\times V) = \ind_{mn}(f\times g, U\times V) \]
as desired.
\end{proof}

Using exactly the same proof as Theorem \ref{fpp}, we have the following solution property:
\begin{thm}\label{productfpp}[Solution property]
If $\iota$ satisfies the additivity property of Theorem \ref{productuniqueness} and $\iota(f\times g, U\times V) \neq 0$, then $\Fix(f\times g) \cap (U\times V)$ is nonempty.
\end{thm}

To prove our product formula we require a lemma about splittings of product maps. Let $\pi_1:X\times Y \to X$ and $\pi_2:X\times Y \to Y$ be the coordinate projections. 

\begin{lem}\label{productsplitting}
Let $f:U\to D_n( X)$ be an $n$-valued map and $g:V\to D_m( Y)$ be an $m$-valued map. Then the product $mn$-valued map $f\times g$ splits if and only if each of $f$ and $g$ split. Specifically, if $\{f_1,\dots,f_n\}$ and $\{g_1,\dots, g_m\}$ are splittings of $f$ and $g$, then $f\times g$ splits as $\{f_i\times g_j\}$, where $i\in \{1,\dots, n\}$ and $j \in \{1,\dots,m\}$. 
\end{lem}
\begin{proof}
It is clear that if each of $f$ and $g$ split, say $f = \{f_1,\dots, f_n\}$ and $g = \{g_1,\dots, g_m\}$, then $\{f_i\times g_j\}$ is a splitting of $f\times g$. 

For the converse, assume that $f\times g$ splits, and we will show that each of $f$ and $g$ splits. Let $\{h_1,\dots,h_{mn}\}$ be a splitting of $f\times g$. Because $f\times g$ is the product of $f$ and $g$ and $f$ is $n$-valued, the set of functions $\{\pi_1 \circ h_k \mid k\in \{1,\dots, mn\} \}$ has exactly $n$ elements, and it is clear that this is a splitting of $f$. Similarly the set of functions given by the various $\pi_2\circ h_k$ has exactly $m$ elements, and this is a splitting of $g$.
\end{proof}

Now we are ready to prove our product formula:
\begin{thm}\label{productthm}
Let $(f,U)$ and $(g,V)$ be admissible pairs. Then $(f\times g, U\times V)$ is an admissible pair, and
\[ \ind_{mn}(f\times g, U\times V) = \ind_n(f,U) \ind_m(g,V). \]
\end{thm}
\begin{proof}
Let $\iota:\mathcal P_{n,m}(X,Y) \to \R$ be defined by $\iota(f\times g, U\times V) = \ind_n(f,U)\ind_m(g,V)$. It suffices to show that $\iota$ satisfies the three axioms of Theorem \ref{productuniqueness}. In that case it must equal the fixed point index $\ind_{mn}$.

To show that $\iota$ satisfies the product homotopy axiom, let $(f,U)$ be admissibly homotopic to $(f',U)$ and $(g,V)$ be admissibly homotopic to $(g',V)$. Then by the homotopy property of the index we have:
\[ \iota(f\times g, U\times V) = \ind_n(f,U) \ind_m(g,V) = \ind_n(f',U) \ind_m(g',V) = \iota(f'\times g', U\times V). \]

To show that $\iota$ satisfies the additivity axiom, let $\Fix(f\times g) \cap (U\times V) \subset (U_1\times V_1) \sqcup (U_2\times V_2)$. Then we will have $\Fix(f) \cap U \subset U_1 \sqcup U_2$ and $\Fix(g) \cap V \subset V_1 \sqcup V_2$, and thus by the additivity property of the $\ind_n$ and $\ind_m$ we have:
\begin{align*} 
\iota(f\times g, U\times V) &= \ind_n(f,U)\ind_m(g,V) \\
&= (\ind_n(f,U_1) + \ind_n(f,U_2))(\ind_m(g,V_1) + \ind_m(g,V_2)) \\
&= \ind_n(f,U_1)\ind_m(g,U_2) + \ind_n(f,U_2)\ind_m(g,U_1) \\
&\qquad+ \ind_n(f,U_2)\ind_m(g,U_1) + \ind_n(f,U_2)\ind_m(g,U_2) \\
&= \iota(f\times g, U_1\times V_1) + \iota(f\times g, U_2\times V_2) \\
&\qquad+ \iota(f\times g, U_1\times V_2) + \iota(f\times g, U_2\times V_1)
\end{align*}
Since $U_1$ and $U_2$ are disjoint and $V_1$ and $V_2$ are disjoint and $\Fix(f\times g) \cap (U\times V) \subset (U_1\times V_1) \sqcup (U_2\times V_2)$, the set $\Fix(f\times g) \cap (U_1\times V_2)$ is empty, and similarly $\Fix(f\times g) \cap(U_2\times V_1)$ is empty. Thus by Theorem \ref{productfpp}, $\iota(f\times g, U_1\times V_2)$ and $\iota(f\times g, U_2\times V_1)$ are zero, and the above becomes
\[ \iota(f\times g, U\times V)  = \iota(f\times g, U_1\times V_1) + \iota(f\times g, U_2\times V_2), \]
and so $\iota$ satisfies the additivity axiom.

Finally we show that $\iota$ satisfies the splitting axiom. Let $f\times g$ split on $U\times V$. By Lemma \ref{productsplitting}, $f$ and $g$ split and the splitting of $f\times g$ has the form $(f_i \times g_j)$, where $(f_i)$ is a splitting of $f$ and $(g_j)$ is a splitting of $g$. To show that $\iota$ satisfies the splitting axiom, assume that $\Fix(f\times g) \cap U\times V = \Fix(f_i\times g_j)\cap U\times V$ is a single point. Then $\Fix(f_i) \cap U$ is a single point, and $\Fix(g_j) \cap V$ is a single point, and we have:
\[
\iota(f\times g,U\times V) = \ind_n(f,U) \ind_m(g,V) = \ind(f_i,U) \ind(g_j,V) = \ind(f_i\times g_j,U\times V)
\]
where the last equality is the product property of the fixed point index $\ind$.
\end{proof}

Taking $U$ and $V$ to be all of $X$ and $Y$ we obtain a product formula for the Lefschetz number:
\begin{thm}
Let $f:X\to D_n( X)$ be an $n$-valued map and $g:Y\to D_m(Y)$ be an $m$-valued map. Then
\[ L(f\times g) = L(f)L(g). \]
\end{thm}

\section{A weaker set of axioms for manifolds}\label{manifoldsection}
In this section we prove a stronger version of Theorem \ref{uniqueness} in the setting of $n$-valued maps on  manifolds. In this case the additional structure on the space allows us to prove our uniqueness subject to a much weaker axiom in place of the splitting axiom. In particular our axioms for the index of an $n$-valued map will make no reference to the classical fixed point index.

We say an $n$-valued map $c:X\to D_n(X)$ is \emph{constant} when there are points $c_1,\dots,c_n \in X$ with $c(x) = \{c_1,\dots,c_n\}$ for all $x\in X$. These points $c_i$ are called the values of $c$.

The following axiom mimics the axiom of the same name used in \cite{fps04}.

\begin{axiom}(Normalization) 
Let $c: X \to D_n(X)$ be a constant $n$-valued map and $U\subset X$ an open set containing exactly one value of $c$. Then $\iota(c,U) = 1$.
\end{axiom}

Our main result in this section uses the uniqueness theorem of \cite{fps04} for the classical fixed point index on $\R^d$. The following is a special case of Theorem 4.2 of \cite{fps04}:
\begin{thm}\label{rnuniqueness}
Let $\iota:\adm(\R^d) \to \R$ be any function satisfying the following axioms:
\begin{itemize}
\item (Homotopy) Let $(f,U), (g,U) \in \adm(\R^d)$ be admissibly homotopic. Then $\iota(f,U) = \iota(g,U)$. 
\item(Additivity) Let $(f,U) \in \adm(\R^d)$ with $\Fix(f) \cap U \subseteq U_1 \sqcup U_2$. Then $\iota(f,U) = \iota(f,U_1) + \iota(f,U_2)$.
\item (Normalization) Let $c$ be a constant map, then $\iota(c,\R^d) = 1$.
\end{itemize}
Then $\iota$ is the local fixed point index.
\end{thm}

We will use a classical fact about the fixed point index on manifolds. Because the index is a local topological invariant, it is preserved by the coordinate charts of the manifold. The following lemma follows from the classical definition of the index, or an axiomatic proof is given as part of the proof of Theorem 4.2 of \cite{fps04}.
\begin{lem}\label{fpslem}
Let $X$ be a complex with $(f,U) \in \adm(X)$, and assume that $U$ is homeomorphic to $\R^d$ for some dimension $d$ by a homeomorphism $\phi:U\to \R^d$. Then the pair $(\phi \circ f \circ \phi^{-1}, \phi(U))$ is admissible in $\R^d$, and 
\[ \ind(f,U) = \ind(\phi \circ f \circ \phi^{-1}, \phi(U)). \]
\end{lem}

Our main result for this section refers to the Homotopy and Additivity axioms from Section \ref{uniquenesssection} (Axioms \ref{htpax} and \ref{addax}):
\begin{thm} $X$ be a manifold. Then the local fixed point index $\ind_n$ is the unique real-valued function on $\adm_n(X)$ satisfying the Homotopy, Additivity, and Normalization axioms.
\end{thm}

\begin{proof}
Let $\iota: \adm_n(X) \to \R$ satisfy the Homotopy, Additivity, and Normalization axioms. We will show that $\iota$ satisfies the Splitting Axiom (Axiom \ref{splittingax}), and therefore by Theorem \ref{uniqueness} it must be the fixed point index $\ind_n$. 

Let $(f,U) \in \adm_n(X)$ such that $f$ splits on $U$ as $\{f_1,\dots,f_n\}$, and let $\Fix(f) \cap U = \Fix(f_i) \cap U$ be a single point $x\in U$. We will show that $\iota(f,U)= \ind(f_i,U)$. To simplify the notation let us assume that $i=1$.

By the excision property, and since $X$ is a manifold, we may assume that $U$ is a set homeomorphic to $\R^d$ where $d$ is the dimension of $X$. Furthermore, since $f(x)$ is a set of $n$ points and $U$ is a neighborhood of $x$, we may assume that $U$ is sufficiently small so that $f(U)$ consists of $n$ disjoint sets $\{f_1(U), \dots, f_n(U)\}$, and $f_j(U) \cap U = \emptyset$ for $j \neq 1$.

Let $\phi:U\to \R^d$ be a homeomorphism. Then for any selfmap $g:\R^d \to \R^d$, let $\bar g:U \to D_n(X)$ be the split $n$-valued map given by $\bar g = \{ \phi^{-1} \circ g \circ \phi, f_2, \dots, f_n\}$. This $\bar g$ is $n$-valued on $U$ since $\phi^{-1}\circ g \circ \phi (U) \subseteq U$ and $f_j(U)$ are each disjoint from $U$ and from one another for $j\neq 1$.

Let $\omega: \adm(\R^d) \to \adm_n(X)$ be defined by 
\[ \omega(g,V) = (\bar g, \phi^{-1}(V)). \]
Then $\iota \circ \omega$ is a real-valued function on $\adm(\R^d)$. It is easy to see that $\iota\circ \omega$ satisfies the three axioms of Theorem \ref{rnuniqueness}, and so it is the fixed point index. 
But $(f,U) = \omega(\phi \circ f_1 \circ \phi^{-1}, \phi(U))$, and so 
\[ \iota(f,U) = \iota(\omega(\phi \circ f_1 \circ \phi^{-1},\phi(U))) = \ind(\phi \circ f_1\circ \phi^{-1}, \phi(U)) = \ind(f_1,U) \]
by Lemma \ref{fpslem}.
\end{proof}

\end{document}